\documentclass{amsart}

\title[Some new inequalities...]{Some new inequalities for the $q$-gamma and related functions}
\usepackage{amssymb,amsmath,amsthm,epsfig,graphics,latexsym}
 \theoremstyle{definition}

  \theoremstyle{plain}
  \newtheorem{lemma}      {Lemma}
  
  \newtheorem{theorem}    {Theorem}
  
  \newtheorem{corollary}  {Corollary}

  \theoremstyle{remark}
  \newtheorem{rmk}{{\bf Remark}}

  \newcommand{\fr}{\frac}

\begin{document}
  \author[M. El Bachraoui and J. S\'{a}ndor]{Mohamed El Bachraoui and J\'{o}zsef S\'{a}ndor}
  \address{Dept. Math. Sci,
 United Arab Emirates University, PO Box 15551, Al-Ain, UAE}
 \email{melbachraoui@uaeu.ac.ae}
 \address{Babes-Bolyai University, Department of Mathematics and Computer Science, 400084 Cluj-Napoca, Romania}
 \email{jsandor@math.ubbcluj.ro}
 \begin{abstract}
We consider convexity and monotonicity properties for some functions related
to the $q$-gamma function. As applications,
we give a variety of inequalities for the $q$-gamma function, the $q$-digamma function $\psi_q(x)$,
and the $q$-series. Among other consequences, we improve a result of Azler~and~Grinshpan about the zeros of the function $\psi_q(x)$.
 We use $q$-analogues for the
Gauss multiplication formula to put in closed form members of some of our inequalities.
 \end{abstract}
  \date{\textit{\today}}
  \maketitle
\section{Introduction}\label{sec-1}
Throughout this paper we assume that $0<q<1$.
The $q$-shifted factorials of a complex number $a$ are defined by
\[
(a;q)_0= 1,\quad (a;q)_n = \prod_{i=0}^{n-1}(1-a q^i),\quad
(a;q)_{\infty} = \lim_{n\to\infty}(a;q)_n.
\]
%
For convenience we write
\[
(a_1,\ldots,a_k;q)_n = (a_1;q)_n\cdots (a_k;q)_n,\quad
(a_1,\ldots,a_k;q)_{\infty} = (a_1;q)_{\infty} \cdots (a_k;q)_{\infty}.
\]
For any complex $x$, we let
\[
[x]_q = \fr{1-q^x}{1-q},
\]
for which we have $\lim_{q\to 1} [x]_q = x$.
The $q$-gamma function is given by
\[
\Gamma_q(x) = \dfrac{(q;q)_\infty}{(q^{x};q)_\infty} (1-q)^{1-x} \quad (|q|<1)
\]
It is clear that
\begin{equation}\label{basic-fact}
\Gamma_q(x+1)= [x]_q \Gamma_q(x)
\end{equation}
and it is well-known that $\Gamma_q (x)$ is a $q$-analogue for the function $\Gamma (x)$, see Askey~\cite{Askey}.
The digamma function is $\psi(x) = \big(\log\Gamma(x)\big)' = \fr{\Gamma'(x)}{\Gamma(x)}$ and its $q$-analogue is the
$q$-digamma function given by
\[
\psi_q (x) = \big(\log\Gamma_q(x)\big)' = \fr{\Gamma_q '(x)}{\Gamma_q (x)}.
\]
The $q$-binomial theorem states that
\begin{equation}\label{q-binomial}
\sum_{n=0}^{\infty}\fr{(a;q)_n}{(q;q)_n} x^n = \fr{(ax;q)_{\infty}}{(x;q)_{\infty}} =: {}_{1}\phi_{0} (a,-;q,x)\quad (|x|<1,\ |q|<1),
\end{equation}
where ${}_{1}\phi_{0} (a,-;q,x)$ is the basic hypergeometric series.
For details and historical notes on the $q$-series, the hypergeometric series, and related functions
we refer to \cite{Andrews-Askey-Roy}.
Our primary goal in this paper is to consider monotonicity and convexity properties of the
$q$-gamma function and some of its related functions.
As an application, we shall present inequalities involving the functions
$\Gamma_q(x)$ and $\psi_q(x)$ along with related functions including
the function ${}_{1}\phi_{0} (a,-;q,x)$. Some of our inequalities involve powers, ratios, and products of
these special functions.
A crucial tool to achieve some of our inequalities is Jensen's inequality stating that if $f(x)$ is a convex
function on $I$ then for all $x_1,\ldots,x_n \in I$ and all positive $a_1,\ldots, a_n$ one has
\[
f\Big(\fr{\sum_{i=1}^n a_i x_i}{\sum_{i=1}^n a_i}\Big) \leq \fr{\sum_{i=1}^n a_i f(x_i)}{\sum_{i=1}^n a_i}.
\]
We mention that refinements of Jensen's inequality exist in literature and thus any inequality we
prove in this paper using Jensen's inequality can appropriately be slightly improved.
For some of other refinements of Jensen's inequality, see
\cite{Pecaric-Svrtan}.
Azler~and~Grinshpan~\cite[Lemma 4.5]{Alzer-Grinshpan} proved that the function $\psi_q(x)$ for
$0<q \not=1$ has a uniquely determined positive zero $x_0 = x_0(q)$.
Among our applications, we shall show that $x_0(q) \in (1,2)$.
We will also provide Ky~Fan type inequalities for the $q$-gamma function.
Another purpose of our work is to establish a variety of inequalities involving the $q$-series.
We note that some of our formulas
have been put in closed forms thanks to $q$-analogues of the Gauss multiplication formula for
the gamma function which we shall describe now.
We recall that the Gaussian multiplication formula for gamma function states that
\begin{equation}\label{Gauss-prod}
\Gamma\left(\fr{1}{n}\right) \Gamma\left(\fr{2}{n}\right) \cdots \Gamma\left(\fr{n-1}{n}\right) =
 \fr{(2\pi)^{\fr{n-1}{2}}}{\sqrt{n}} \quad (n=1,2,\ldots).
\end{equation}
A famous $q$-analogue for (\ref{Gauss-prod}) due to Jackson~\cite{Jackson-1, Jackson-2},
 states that
 \begin{equation}\label{Jackson-q-Gamma}
 \left(\fr{1-q^n}{1-q}\right)^{nz-1} \Gamma_{q^n} (z)\Gamma_{q^n} \left(z+\fr{1}{n}\right)
 \cdots \Gamma_{q^n} \left(z+\fr{n-1}{n}\right)
 \end{equation}
 \[
 =
 \Gamma_{q^n} (nz)\Gamma_{q^n} \left(\fr{1}{n}\right) \Gamma_{q^n} \left(\fr{2}{n}\right)
 \cdots \Gamma_{q^n} \left(\fr{n-1}{n}\right) \quad (n=1,2,\ldots).
 \]
Recently, the authors~\cite{Bachraoui-Sandor} gave the following $q$-analogue for
(\ref{Gauss-prod})
\begin{equation}\label{q-Gauss-prod-1}
\prod_{k=1}^{n-1}\Gamma_q\left(\fr{k}{n}\right) =
\Big(\Gamma_q\big(\fr{1}{2}\big)\Big)^{n-1} \fr{(q^{\fr{1}{2}}; q^{\fr{1}{2}})_{\infty}^{n-1}}
{(q;q)_{\infty}^{n-2} (q^{\fr{1}{n}};q^{\fr{1}{n}})_{\infty}}.
\end{equation}
Besides, S\'{a}ndor~and~T\'{o}th~\cite{Sandor-Toth} found
 \begin{equation}\label{short-gam-prod}
 P(n):= \prod_{\substack{k=1 \\ (k,n)=1}}^n \Gamma\left(\fr{k}{n}\right)
 = \fr{(2\pi)^{\fr{\varphi(n)}{2}}}{ \prod_{d\mid n} d^{\fr{1}{2}\mu\left(\fr{n}{d}\right)} }
 = \fr{(2\pi)^{\fr{\varphi(n)}{2}}}{ e^{\fr{\Lambda(n)}{2}} },
 \end{equation}
 where $\varphi(n)$ in the Euler totient function, $\mu(n)$ is the M\"{o}bius mu function, and $\Lambda(n)$ is the Von~Mangoldt function.
 We accordingly let
 \[
 P_q(n) = \prod_{\substack{k=1 \\ (k,n)=1}}^n \Gamma_q\Big(\fr{k}{n}\Big).
 \]
The authors~\cite{Bachraoui-Sandor} also found the following $q$-analogue
(\ref{short-gam-prod}).
\begin{equation}\label{q-short-gam-prod}
P_q(n) = \fr{ \left(\Gamma_q\Big( \fr{1}{2} \Big)\right)^{\varphi(n)} (q^{\fr{1}{2}};q)^{\varphi(n)}}
{ \prod_{d\mid n}\big( \big(q^{\fr{1}{d}};q^{\fr{1}{d}}\big)_{\infty} \big)^{\mu\left(\fr{n}{d}\right)}} 
= \fr{\left( 2\Gamma_q\Big(\fr{1}{2}\Big)\right)^{\fr{\varphi(n)}{2}}}{e^{\fr{\Lambda_q(n)}{2}}}.
\end{equation}
\section{Inequalities for $\psi_q(x)$ and $\Gamma_q(x)$}\label{sec-2}
\begin{lemma}\label{lem-1}
\emph{(a)}\ The derivative of the function $\psi_q(x)$ is strictly completely monotonic on $(0,\infty)$, that is,
\[
(-1)^n \big(\psi'(x) \big)^{(n)} >0  \quad x>0,\ n=0,1,2,\ldots.
\]
\noindent
\emph{(b)}\ For any $x\geq 1$, we have that $x \big(\psi_q(x) \big)' + 2 \psi_q(x) >0$.
\end{lemma}
\begin{proof}
Part (a) is an immediate consequence of the series representation
\[
\psi_q(x) = -\log(1-q)+(\log q)\sum_{n=0}^{\infty}\fr{q^{n+x}}{1-q^{n+x}}
= -\log(1-q)+(\log q)\sum_{n=1}^{\infty}\fr{q^{nx}}{1-q^{n}}.
\]
Part (b) is due Azler~and~Grinshpan~\cite[Lemma 3.4]{Alzer-Grinshpan}.
\end{proof}
\begin{lemma}\label{lem-2}
For all $x>0$, we have
\[
\fr{q^x \log q}{1-q^x} + \log[x]_q  < \psi_q (x) < \log[x]_q.
\]
\end{lemma}
\begin{proof}
From (\ref{basic-fact}), we deduce that $\log \Gamma_q (x+1) - \log \Gamma_q (x) = \log [x]_q$. Then by Lagrange mean value theorem, there exists $t\in(0,1)$
such that
\begin{equation}\label{help-1-lem-2}
\psi_q(x+t) = \log [x]_q.
\end{equation}
As $\psi_q (x)$ is strictly increasing by Lemma~\ref{lem-1}, the forgoing identity implies that
\begin{equation}\label{help-2-lem-2}
\psi_q(x) < \psi_q(x+t) < \psi_q(x+1).
\end{equation}
Next, differentiate both sides of (\ref{basic-fact}) to obtain
\begin{equation}\label{help-3-lem-2}
\psi_q(x) =  \fr{q^x \log q}{1-q^x} + \psi_q(x+1).
\end{equation}
Now, combine (\ref{help-1-lem-2}), (\ref{help-2-lem-2}), and (\ref{help-3-lem-2}) to get the desired inequalities.
\end{proof}
\noindent
Azler~and~Grinshpan~\cite[Lemma 4.5]{Alzer-Grinshpan} proved that the function $\psi_q(x)$ for $0<q \not=1$ has a uniquely determined positive zero $x_0 = x_0(q)$.
For  $0<q<1$, it turns out that $x_0(q) \in (0,1)$ as we will see now.
\begin{theorem}\label{thm-psi-q-zero}
(a)\ The function $\psi_q(x)$ has a unique zero $x_0$ in the interval $(1,2)$. \\
\noindent
(b)\ There holds $\Gamma_q(x)\geq \Gamma_q(x_0)$ for all $x\in (0,\infty)$.
\end{theorem}
\begin{proof}
\emph{First proof of (a)}\  Application of Lemma~\ref{lem-2} to $x=1$ and to $x=2$ respectively gives
\begin{equation}\label{psi-q-1}
\fr{q\log q}{1-q} < \psi_q(1) < 0 \quad\text{and\quad}
\fr{q^2\log q}{1-q^2} + \log(1+q) < \psi_q(2).
\end{equation}
By the well-known fact that the function $\psi(x)$ is strictly increasing and continuous we will be done
if we show that
\begin{equation}\label{psi-q-2}
\fr{q^2\log q}{1-q^2} + \log(1+q) >0.
\end{equation}
Letting $q= \fr{1}{t}$ for $t>1$ and after simplification
(\ref{psi-q-2})  becomes
\[
\log t < \log(t+1) - \fr{\log t}{t^2-1},
\]
or equivalently,
\[
(t^2-1) \log (t+1) -t^2 \log t > 0.
\]
Letting $f(t)= (t^2-1) \log (t+1) -t^2 \log t$, we find that
\[
\begin{split}
f'(t) &= 2t \log(t+1) -2t \log t -1 \\
f''(t) &= 2\left(\log\Big(1+\fr{1}{t}\Big) - \fr{1}{t+1} \right)
\end{split}
\]
Then by a combination of the previous identity and the well-known inequality
$\Big( 1+\fr{1}{t}\Big)^{t+1} >e$,
we deduce that $f''(t)>0$ from which it follows that $f'(t)$ is strictly increasing. Then from the above,
$f'(t)> f'(1)  = 2 \log 2 -1 >0$, which in turn shows that $f(t)$ is strictly increasing. Therefore
$f(t) > f(1) =0$, establishing the relation~(\ref{psi-q-2}).
\\
\noindent
\emph{Second proof of (a)}\ As $\Gamma_q(1)=\Gamma_q(2) = 1$, we have by Rolle's theorem applied to $\Gamma_q(x)$ on $[1,2]$
there exists $x_0\in (1,2)$ such that $\big(\Gamma_q(x_0)\big)' = 0$  and hence $\psi_q(x_0)=0$ as $\big(\Gamma_q(x)\big)' = \psi_q(x) \Gamma_q(x)$. Since $\Gamma_q(x)$ is
strictly convex, its derivative is strictly increasing, and so $x_0$ is
unique.
\\
\noindent
(b)\ It is well-known that a strict log-convex function is also strict convex and so, $\Gamma_q(x)$ is strict convex on $(0,\infty)$ by Lemma~\ref{lem-1}(a). That is,
$\big(\Gamma_q(x)\big)'$ is strictly increasing on $(0,\infty)$.
Now Combine this with the identity $\big(\Gamma_q(x)\big)' = \psi_q(x) \Gamma_q(x)$ as follows. Then $\big(\Gamma_q(x)\big)' < \big(\Gamma_q(x_0)\big)' =0$ on the left of $x_0$
and $\big(\Gamma_q(x)\big)' \geq \big(\Gamma_q(x_0)\big)' =0$ on the right of $x_0$, showing that the function
$\Gamma_q(x)$ is strictly decreasing on $(0,x_0)$ and strictly increasing on $(x_0,\infty)$. This completes the proof.
\end{proof}
\begin{rmk}\label{rmk-0}
By the known inequality $\log x < x-1$ for $1\not= x >0$ applied to $x= \fr{1}{q}$, we get
$\fr{q\log q}{q} >-1$ and so by  (\ref{psi-q-1}), we get
\begin{equation}\label{psi-q-3}
-1 < \psi_q(1) < 0.
\end{equation}
\noindent
Bradley~\cite{Bradley} introduced an extension of the Euler gamma constant $\gamma$ as follows
\[
\gamma_q = \log (1-q) - \fr{\log q}{1-q}\sum_{i=1}^{\infty} \fr{q^i}{[i]_q}
\]
and proved that $\lim_{q\to 1} \gamma_q = \gamma$. Mahmoud and Agarwal~\cite{Mahmoud-Agarwal} proved that for $0<q<1$ we have
$\psi_q(1) = \gamma_q$. We note that there no any other information in \cite{Bradley} and \cite{Mahmoud-Agarwal} related
to the generalized constant $\gamma_q$. From (\ref{psi-q-3}), it follows that
\[
0<\gamma_q < 1\ \text{for any\ } q\in(0,1).
\]
Stronger approximations are given in \cite{Bachraoui-Sandor-2}.
\end{rmk}
\begin{theorem}\label{thm-q-gam-psi-zero}
(a)\ The function $\log \Gamma_q(x) + x\psi_q(x)$ is strictly increasing on $(1,\infty)$ with a single
zero which is in $(1,2)$.

\noindent
(b)\  The function $\log \Gamma_q(x) - x\psi_q(x)$ is strictly decreasing on $(0,\infty)$ with a single
zero which is in $(1,2)$.
\end{theorem}
\begin{proof}
Let $f(x) = \log\Gamma_q(x)+ x\psi_q(x)$. Then by Lemma~\ref{lem-1}(b),
$f'(x) = 2\psi_q(x) + x\psi_q'(x) >0$ and therefore $f(x)$ is strictly increasing on $(1,\infty)$. We have
already seen in the proof of Theorem~\ref{thm-psi-q-zero} that $\psi_q(1)<0<\psi_q(2)$. It follows that
$f(1) = \psi_q(1) <0 < 2\psi_q(2) = f(2)$. As the function $f(x)$ is clearly continuous on $(1,\infty)$,
the proof is complete for part (a). Part (b) follows in exactly the same way.
\end{proof}
\begin{corollary}\label{cor-telescoping}
For any $x>1$ and any positive integer $n$ we have the following double inequality
\[
x\psi_q(x) - (x+n)\psi_q(x+n) < \log\fr{\Gamma_q(x)}{\Gamma_q(x+n)} < (x+n)\psi_q(x+n) - x\psi_q(x).
\]
\end{corollary}
\begin{proof}
By Theorem~\ref{thm-q-gam-psi-zero}(a), we have
$\log\Gamma_q(x) + x\psi_q(x) < \log\Gamma_q(y)+y\psi_q(y)$ whenever $1<x<y$. Repeatedly
application of this and simplifying yield
\[
\begin{split}
\log\Gamma_q(x)- \log\Gamma_q(x+1) & <   (x+1)\psi_q(x+1)- x\psi_q(x) \\
\log\Gamma_q(x+1) - \log\Gamma_q(x+2) & < (x+2)\psi_q(x+2) - (x+1)\psi_q(x+1) \\
&\vdots \\
\log\Gamma_q(x+n-1)-\log\Gamma_q(x+n) & < (x+n)\psi_q(x+n) - (x+n-1)\psi_q(x+n-1).
\end{split}
\]
Adding together gives
\[
\log \Gamma_q(x)-\log\Gamma_q(x+n) < -x\psi_q(x)+ (x+n)\psi_q(x+n),
\]
which is equivalent to the first inequality. The second inequality is obtained similarly by considering
the function $\log\Gamma_q(x) - x\psi_q(x)$ which is decreasing by
Theorem~\ref{thm-q-gam-psi-zero}(a).
\end{proof}
\begin{rmk}\label{rmk-1}
It is known, see for instance \cite{Askey}, that the function $x\psi_q (x)$ is strictly convex, so, as $\log \Gamma_q(x)$
is strictly convex too, we get that the function $f(x)=\log\Gamma_q(x)+x\psi_q(x)$ of
Theorm~\ref{thm-q-gam-psi-zero}(a) is strictly convex.
\end{rmk}
\begin{lemma}\label{lem-q-psi-sum}
For any positive integers $k$ and $n$ there holds
\[
\begin{split}
(a)\ & \sum_{i=1}^{n-1} \psi_q\big(\fr{i}{n}\big) = (n-1)\psi_q(1) - n\log\fr{1-q}{1-q^{\fr{1}{n}}} \\
(b)\ & \sum_{i=1}^{n-1} \psi_q^{(k)} \big(\fr{i}{n}\big) = (n^{k+1}-1)\psi_q^{(k)}(1).
\end{split}
\]
\end{lemma}
\begin{proof}
In (\ref{Jackson-q-Gamma}) replacing $q^n$ with $q$ and taking logarithms on both sides give
\begin{equation}\label{Jackson-q-derive}
(nz-1)\log\fr{1-q}{1-q^{\fr{1}{n}}} + \sum_{i=0}^{n-1}\log\Gamma_q (z+\fr{i}{n})
= \log \Gamma_q (nz) + \log\prod_{i=1}^{n-1}\Gamma_q\big(\fr{i}{n}\big).
\end{equation}
Differentiating with respect to $z$ and then letting $z=\fr{1}{n}$ yield
\[
n\log\fr{1-q}{1-q^{\fr{1}{n}}} + \sum_{i=1}^{n-1}\psi_q\big(\fr{i}{n}\big)  + \psi_q(1) = n\psi_q(1),
\]
which is equivalent to the desired identity in part (a). To prove part (b), first differentiate with respect $z$,
$k$ times both sides of (\ref{Jackson-q-derive}) to obtain
\[
\sum_{i=0}^{n-1} \psi_q^{(k)} \big(z+\fr{i}{n}\big) = n^{k+1} \psi_q^{(k)} (nz),
\]
then let $z=\fr{1}{n}$ to get
\[
\sum_{i=1}^{n-1} \psi_q^{(k)} \big(\fr{i}{n}\big) = (n^{k+1}-1) \psi_q^{(k)} (1),
\]
as desired.
\end{proof}
\begin{theorem}\label{thm-q-psi-sum}
For any positive integers $k$ and $n$ there holds
\[
\begin{split}
(a)\ & (n-1) \Big(\psi_q(1) - \psi_q\big(\fr{1}{2}\big) \Big)< n \log\fr{1-q}{1-q^{\fr{1}{n}}} \\
(b)\ & (n-1) \psi_q^{(2k-1)}\big(\fr{1}{2}\big) < (n^{2k}-1) \psi_q^{(2k-1)} (1) \\
(c)\ & (n-1) \psi_q^{(2k)}\big(\fr{1}{2}\big) > (n^{2k+1}-1) \psi_q^{(2k)} (1).
\end{split}
\]
\end{theorem}
\begin{proof}
(a)\ By Lemma~\ref{lem-1}(a), the function $\psi_q(x)$ is strictly concave. Then by an application of
Jensen's inequality to this function with $k=n-1$ and $x_i = \fr{i}{n}$ for $i=1,\ldots,n-1$ we find
\[
\psi_q\Big(\fr{\sum_{i=1}^{n-1}\fr{i}{n}}{n-1} \Big) > \fr{1}{n-1} \sum_{i=1}^{n-1}\psi_q\big(\fr{i}{n}\big).
\]
Then by an appeal to Lemma~\ref{lem-q-psi-sum}(a) along with simplification we derive
\[
(n-1)\psi_q\big(\fr{1}{2}\big) >  (n-1) \psi_q(1) - n\log\fr{1-q}{1-q^{\fr{1}{n}}},
\]
which proves part (a).

\noindent
(b)\ By Lemma~\ref{lem-1}(a), the function $\psi_q^{(2k-1)}(x)$ is strictly convex and therefore by
Jensen's inequality applied to this function with $k=n-1$ and $x_i = \fr{i}{n}$ for $i=1,\ldots,n-1$ one has
\[
\psi_q^{(2k-1)} \Big(\fr{\sum_{i=1}^{n-1}\fr{i}{n}}{n-1} \Big) <
\fr{1}{n-1} \sum_{i=1}^{n-1}\psi_q^{(2k-1)}\big(\fr{i}{n}\big).
\]
Now use Lemma~\ref{lem-q-psi-sum}(b) and simplify to deduce that
\[
(n-1) \psi_q^{(2k-1)}\big(\fr{1}{2}\big) < (n^{2k}-1) \psi_q^{(2k-1)}(1),
\]
which is the desired relation in part (b). The similar proof of part (c) is omitted.
\end{proof}
\section{Convexity and inequalities for powers, ratios, and products of $\Gamma_q(x)$}\label{sec-3}
\begin{lemma}\label{lem-3-1}
Let $f:(0,\infty)\to (0,\infty)$ and let $g(x)= \fr{f(x+1)}{f(x)}$. If $f(x)$ is strictly log-convex on $(0,\infty)$, then
for any $x>0$ and any $a\in (0, 1)$ we have
\[
\big( g(x) \big)^{1-a} < \fr{f(x+1)}{f(x+a)} < \big(g(x+a) \big)^{1-a}.
\]
\end{lemma}
\begin{proof}
$f(x)$ is strictly log-convex on $(0,\infty)$ we have for any $u\in[0,1]$ and any $y\not=z >0$
\[
\log \big(f(uy + (1-u) z) \big) < u\log f(y) + (1-u) \log f(z)
\]
or equivalently,
\begin{equation}\label{help-1-lem-3-1}
f(uy + (1-u) z) < \big( f(y) \big)^u  \big( f(z) \big)^{1-u}.
\end{equation}
Let in (\ref{help-1-lem-3-1}) $y:=x$, $z:= x+1$, and $u:= 1-a$ to obtain
\[
f(x+a) < \big(f(x)\big)^{1-a} \big( f(x+1) \big)^a,
\]
from which we easily get the first inequality.
As to the second inequality, let in (\ref{help-1-lem-3-1}) $y:= x+a$, $z:= x+a+1$, and $u:= a$ and proceed as before.
\end{proof}
A classical result by Gautschi~\cite{Gautschi} states that
\[
x^{1-a} < \fr{\Gamma(x+1)}{\Gamma(x+a)} < e^{(1-a)\psi(x+1)} \quad (0<a<1).
\]
We have the following $q$-variant which seems to be new.
\begin{corollary}
Let $x>0$ and $a\in(0,1)$. Then
\[
\big( [x]_q \big)^{1-a} < \fr{\Gamma_q(x+1)}{\Gamma_q(x+a)} < \big( [x+a]_q \big)^{1-a}.
\]
\end{corollary}
\begin{proof}
Simply apply Lemma~\ref{lem-3-1} to the function $f(x)= \log \Gamma_q(x)$.
\end{proof}
\noindent
The following result is well-known.
\begin{lemma}\label{lem-3-2}
Let $I\subseteq \mathbb{R}$ and let $f:I \to (0,\infty)$.

\noindent
(a)\ If $f(x)$ is concave (strictly concave), then
$\fr{1}{f(x)}$ is convex (strictly convex).

\noindent
(b)\ The function $f(x)$ is log-convex if and only if $\fr{1}{f(x)}$ is log-concave.
\end{lemma}
Note that the converse of Lemma~\ref{lem-3-2}(a) is not true. For example, the function $e^x$ is convex but the reciprocal $e^{-x}$ is not concave.
\begin{lemma}\label{lem-3-2-1}
For any $x>0$ we have
\[
\begin{split}
\emph{(a)\ } \big(\psi_q(x+1)\big)' < \fr{-q^x \log q}{1-q^x} < \big( \psi_q(x) \big)' \\
\emph{(b)\ } \big(\psi_q(x)\big)'' < \fr{-q^x (\log q)^2}{(1-q^x)^2} < \big( \psi_q(x+1) \big)''.
\end{split}
\]
\end{lemma}
\begin{proof}
By the Lagrange mean value theorem there is $y\in(x,x+1)$ such that
$\psi_q(x+1) - \psi_q(x) = \big( \psi_q(y) \big)'$. As the the function
$\big(\psi_q(x) \big)'$ is strictly decreasing by Lemma~\ref{lem-1}(a), a combination of the
previous identity with the relation (\ref{help-3-lem-2}) yields part (a). To establish part (b),
note first that using the Lagrange mean value theorem there exists $z\in(x,x+1)$ such that
$\psi_q(x+1) - \psi_q(x) = \big( \psi_q(z) \big)''$. Note also that by (\ref{help-3-lem-2}) one has
\[
\big(\psi_q(x+1)\big)' - \big(\psi_q(x)\big)' = \fr{-q^x (\log q)^2}{(1-q^x)^2}.
\]
Moreover, the function $\big(\psi_q(x) \big)'$ is strictly increasing by Lemma~\ref{lem-1}(a).
Thus part (b) follows by a combination of the above facts.
\end{proof}
\begin{corollary}\label{cor-add-1}
There holds
\[
\begin{split}
\emph{(a)\ } \big(\psi_q(x+1)\big)' & < \fr{1}{x}, \ \text{for all\ } x>0 \\
\emph{(b)\ } \big( \psi_q(x+1)\big)'' & > \fr{-1}{x^2}, \ \text{for all\ } x>0.
\end{split}
\]
\end{corollary}
\begin{proof}
Upon letting $t=q^x$ and using Lemma~\ref{lem-3-2-1}, we see that to prove part (a)
it will be enough to show that $\fr{t \log t}{1-t} >-1$. But this inequality has been
established in Remark~\ref{rmk-0}. Similarly, but now letting $q^x = \fr{1}{p}$ for $p>1$, we can check
that in order to prove part (b),
it will be enough to show that
\[
\sqrt{p}<\fr{p-1}{\log p}.
\]
To this end, let $A(a,b)$ be the arithmetic
mean, $G(a,b)$ be the geometric mean, and $L(a,b)$ be the logarithmic mean, i.e.
\[ A(a,b) = \fr{a+b}{2},\quad G(a,b) = \sqrt{ab},\quad L(a,b) = \fr{b-a}{\log b - \log a}.
\]
It is well-known (see for instance S\'{a}ndor~\cite{Sandor-1990, Sandor-2016}) for these means that
\begin{equation}\label{means}
G(a,b)<L(a,b)<A(a,b).
\end{equation}
In particular, we have that $G(p,1)< L(p,1)$, which is the desired inequality.
\end{proof}
\begin{theorem}\label{thm-1}
The function $f(x)= \big(\Gamma_q(x+1)\big)^{\fr{1}{x}}$ is strictly log-concave and strictly increasing
on $(0,\infty)$.
\end{theorem}
\begin{proof}
We find
\begin{equation}\label{help-thm-1}
x^3 \big(\log f(x) \big)'' = x^2 \big(\psi_q(x+1) \big)' - 2x \psi_q(x+1) + 2 \log \Gamma_q (x+1)
\end{equation}
and letting $h(x) = x^2 \big(\psi_q(x+1) \big)' - 2x \psi_q(x+1) + 2 \log \Gamma_q (x+1)$, we have
$h'(x) = x^2 \big(\psi_q(x+1)\big)''$ and so $h'(x) <0$ by Lemma~\ref{lem-1}(a), that is, $h(x)$ is strictly decreasing on $(0,\infty)$. Then $h(x) < h(0)=0$, which combined
with equation (\ref{help-thm-1}) implies that $\log f(x)$ is strictly concave on $(0,\infty)$ and therefore, the first statement follows. As to the monotonicity, observe that
$\big(\log f(x) \big)' = \fr{a(x)}{x^2}$, where
\[
a(x)= x\psi_q(x+1)-\log \Gamma_q(x+1).
\]
The function $a(x)$ is well-defined on $[0,\infty)$. One clearly has $a(0)= 0$ and by
Lemma~\ref{lem-1},
$a'(x) = x \big(\psi_q(x)\big)' >0$. Thus $a(x)>a(0)=0$ and we have that the function
$f(x)$ is strictly increasing on $(0,\infty)$.
\end{proof}
\begin{corollary}\label{cor-3-0}
Let $f(x)$ be the function in Theorem~\ref{thm-1}. Then for any $x>0$ we have
\[
e^{-\gamma_q}< f(x) < \fr{1}{1-q}.
\]
\end{corollary}
\begin{proof}
Since $f(x)$ is strictly increasing by Theorem~\ref{thm-1}, one has
\[
f(0^{+}) < f(x) < f(\infty).
\]
Recall further that $\psi_q(1) = -\gamma_q$ and note
 that by Azler and Grinshpan~\cite{Alzer-Grinshpan} we have
$\lim_{x\to\infty}\psi_q(x) = -\log (1-q)$.
Now, by l'Hopital's rule we find that
\[
\lim_{x\to\infty} \log f(x) = \psi_q(1) = -\gamma_q \ \text{and\ }
\lim_{x\to\infty} \log f(x) = \lim_{x\to\infty}\psi_q(x+1) = -\log(1-q),
\]
which yields the desired inequalities.
\end{proof}
\begin{corollary}\label{cor-3-1}
The function $\fr{1}{\Gamma_q(x+1)^{\fr{1}{x}}}$ is strictly log-convex on $(0,\infty)$.
\end{corollary}
\begin{proof}
This follows by Theorem~\ref{thm-1} and Lemma~\ref{lem-3-2}(b).
\end{proof}
A classical gamma version for Corollary~\ref{cor-3-1} is due to Van~de~Lune~\cite{Vandelune}.
\begin{corollary}\label{cor-3-2}
Let
\[
g(x) = \left(\fr{\Big(\Gamma_q(x+1)\Big)^{\fr{1}{x}}}{[x+1]_q} \right)^{\fr{1}{x+1}}.
\]
 Then for any $x>0$ and any $a\in (0,1)$, we have
\[
\big(g(x)\big)^{1-a} < \fr{\Big(\Gamma_q(x+a+1)\Big)^{\fr{1}{x+a}} }{\Big(\Gamma_q(x+2)\Big)^{\fr{1}{x+1}}} < \big(g(x+a)\big)^{1-a}.
\]
\end{corollary}
\begin{proof}
This is a direct consequence of Corollary~\ref{cor-3-1} and Lemma~\ref{lem-2} and the basic fact that
$\Gamma_q(x+2) = [x+1]_q \Gamma_q(x+1)$.
\end{proof}
\begin{theorem}\label{thm-1-1}
Let $f(x)$ be as defined in Theorem~\ref{thm-1}. Then the function $F(x)=\fr{f(x)}{x}$ is strictly
decreasing and strictly log-convex on $(0,\infty)$.
\end{theorem}
\begin{proof}
Let $b(x) = \log F(x)$ and $c(x)= x^2 b'(x)$. Then we easily find that
$b(x) = -\log x + \fr{1}{x}\log\Gamma_q(x+1)$ and therefore that
$c(x) = -x-\log \Gamma_q(x+1)+x\psi_q(x+1)$.
Note that the function $c(x)$ can be defined for all $x\geq 0$ and we have that
$c(0)=0$ and $c'(x) = -1 + x \big(\psi_q(x+1)\big)'$. The last identity and Corollary~\ref{cor-add-1}(a)
imply that $c'(x) <0$,  thus $c(x)<c(0)$, so $b'(x)<0$, and consequently the function $b(x)$ is
strictly decreasing. Now we consider the convexity of $b(x)$. As we have seen above, one has
\[
b'(x) = \fr{-1}{x}- \fr{1}{x^2} \log \Gamma_q(1+x) + \fr{1}{x} \psi_q(1+x),
\]
from where we obtain, by letting $d(x)= x^3 b''(x)$, that
\[
d(x)= x+ 2\log \Gamma_q(x+1) - 2x \psi_q(x+1) + x^2 (\psi_q(x+1))'
\]
The function can be defined on $[0,\infty)$. One has $d(0)=0$ and we obtain after some computations
that $d'(x) = 1 + x^2 \big(\psi_q(x+1) \big)''$. Then by virtue of Corollary~\ref{cor-add-1}(b) we have
that $d'(x) >0$. Thus $d(x)>d(0) = 0$, and the result follows.
\end{proof}
\begin{rmk}
The two monotonicity properties of the functions $f(x)$ of Theorem~\ref{thm-1} and of
$\fr{f(x)}{x}$ of Theorem~\ref{thm-1-1} have been proved for the case of the classical gamma function
by Kershaw and Laforgia~\cite{Kershaw-Laforgia}. We note also that in~\cite{Kershaw-Laforgia}
it was proved that the function $\Gamma\big(1+\fr{1}{x}\big)^x$ is strictly decreasing and
$x \Gamma\big(1+\fr{1}{x}\big)^x$ is strictly increasing. It is immediate that these are equivalent
with the above monotonicity theorems for $q=1$.
\end{rmk}
\begin{theorem}\label{thm-2}
(a)\ The function $\Big( \Gamma_q(x) \Big)^{\fr{1}{x}}$ is strictly log-convex on $(0,1]$.

\noindent
(b)\ The function $\Big(\Gamma_q(x) \Big)^{x}$ is strictly log-convex on $[1,\infty)$.
\end{theorem}
\begin{proof}
(a)\ Letting $f(x) = \Big( \Gamma_q(x) \Big)^{\fr{1}{x}}$, we get
\begin{equation}\label{help-thm-2}
x^3 \big(\log f(x) \big)'' = x^2 \big(\psi_q(x) \big)' - 2x \psi_q(x) + 2 \log \Gamma_q (x+1).
\end{equation}
Now for the function $h(x)= x^2 \big(\psi_q(x) \big)' - 2x \psi_q(x) + 2 \log \Gamma_q (x+1)$ we find $h'(x) = x^2 \big(\psi_q(x) \big)'' <0$, that is, the function
$h(x)$ decreases on $(0,1]$. Then
for any $x\in (0,1]$ we have with the help of Lemma~\ref{lem-1}(a) and  inequality~(\ref{psi-q-1})
\[
h(x) \geq h(1) = \psi_q'(1) - 2 \psi_q(1) >0.
\]
Thus $\big(\log f(x) \big)''>0$ on $(0,1]$, or equivalently, $f(x) = \Big( \Gamma_q(x) \Big)^{\fr{1}{x}}$ is log-convex on $(0,1]$.
(b)\ We have by a straight computation and Lemma~\ref{lem-1}(b),
\[
\Big(\log \big(\Gamma_q(x) \big)^{x}\Big)'' = 2 \psi_q(x) + x \big(\psi_q(x) \big)' >0,
\]
showing the desired statement.
\end{proof}
We note that Theorem~\ref{thm-1} and Theorem~\ref{thm-2}(a)  were motivated by results of the second author in ~\cite{Sandor-2, Sandor-3} on Euler gamma function.
\begin{lemma}\label{lem-3-3}
Let $f(x)$ be strictly log-convex on the interval $(0,1)$. Then we have
\[
\fr{f\big(1-2x(1-x)\big)}{f(1-x)} < \Big(\fr{f(x)}{f(1-x)} \Big)^x < \fr{f(x)}{f\big( 2x(1-x) \big)}.
\]
\end{lemma}
\begin{proof}
As $f(x)$ is strictly log-convex we have for any $a\in(0,1)$
\[
f\big( a(1-x) + (a-1) x \big) < \big(f(1-x)\big)^a \big( f(x) \big)^{1-a},
\]
which by letting $a=1-x$ means
\[
f(2x^2 - 2x +1) < \big(f(1-x)\big)^{1-x} \big( f(x) \big)^{x}.
\]
It follows that
\[
\fr{f\big(1- 2x(1-x)\big)}{f(1-x)} < \Big(\fr{f(x)}{f(1-x)} \Big)^x,
\]
which is the first desired inequality.
As to the second inequality, as $f(x)$ is strictly log-convex we also have for any $a\in(0,1)$
\[
f\big( a x + (a-1)(1-x) \big) < \big(f(x)\big)^a \big( f(1-x) \big)^{1-a},
\]
which by letting $a=1-x$ means
\[
f\big( 2x(1-x) \big) < \big(f(x)\big)^{1-x} \big( f(1-x) \big)^{x},
\]
or equivalently,
\[
\fr{f(x)}{f\big( 2x(1-x) \big)} > \Big(\fr{f(x)}{f(1-x)} \Big)^x.
\]
This completes the proof.
\end{proof}
\begin{corollary}\label{cor-3-3}
For any $x\in(0,1)$ we have
\[
\begin{split}
(a)\quad & \Gamma_q\big( 2x(1-x) \big) \Gamma_q\big(1-2x(1-x) \big) < \Gamma_q (x) \Gamma_q(1-x) \\
(b)\quad & \psi_q\big( 2x(1-x) \big) \psi_q\big(1-2x(1-x) \big) > \psi_q (x) \psi_q(1-x).
\end{split}
\]
\end{corollary}
\begin{proof}
Let $f(x) = \Gamma_q(x) \Gamma_q (1-x)$. Then $f(x)$ is strictly log-convex being the product of two log-convex functions and we clearly have
$f(x)= f(1-x)$. Then by virtue of Lemma~\ref{lem-3-3}, we get
\[
\fr{\Gamma_q\big(2x(1-x) \big)\Gamma_q\big(1-2x(1-x) \big)}{\Gamma_q(x)\Gamma_q(1-x)} < 1 < \fr{\Gamma_q(x)\Gamma_q(1-x)}{\Gamma_q\big(2x(1-x) \big)\Gamma_q\big(1-2x(1-x) \big)}.
\]
Then simplifying gives part (a). As to part (b), let $g(x) = \psi_q(x) \psi_q (1-x)$. Then by Lemma~\ref{lem-1}(a) we find
\[
\big(\log g(x) \big)'' = \big(\psi(x)\big) '' + \big(\psi(1-x)\big) '' <0,
\]
showing by Lemma~\ref{lem-3-2} that the reciprocal $\fr{1}{g(x)}$ is strictly log-convex. Moreover, it is clear that $g(x)= g(1-x)$. Then
from Lemma~\ref{lem-3-3}, we have
\[
\fr{\psi_q(x)\psi_q(1-x)}{\psi_q\big(2x(1-x)\big) \psi_q\big(1-2x(1-x)\big)} <1.
\]
This completes the proof.
\end{proof}
%
For our next result, we need the following result of Vasi\'{c}~\cite{Vasic} which is an extension of
a famous inequality of Petrovi\'{c}. We refer to \cite{Mitrinovic} for details about Petrovi\'{c}'s inequality.
\begin{lemma}\label{lem-Petrovic}
Let  $f:[0,\infty) \to \mathbb{R}$ be convex. Then for any $x_1,\ldots,x_n \geq 0$
and any $p_1,\ldots, p_n \geq 1$, we have
\[
\sum_{i=1}^n p_i f(x_i) \leq  f\big(\sum_{i=1}^n p_i x_i\big) + \big(\sum_{i=1}^{n-1}p_i -1 \big)f(0).
\]
\end{lemma}
\begin{corollary}\label{cor-Petrovic}
For any real numbers $x_1,\ldots,x_n \geq 0$, we have
\[
\prod_{i=1}^n \Gamma_q(x_i) \leq \fr{[\sum_{i=1}^n x_i]_q}{\prod_{i=1}^n [x_i]_q}
\Gamma_q\Big(\sum_{i=1}^n x_i \Big).
\]
\end{corollary}
\begin{proof}
Simply apply Lemma~\ref{lem-Petrovic} to $f(x)= \log\Gamma_q(x+1)$, $p_1=\ldots=p_n=1$ and note
that $f(0)=0$.
\end{proof}
\section{Inequalities related to $\fr{\Gamma_q(1-x)}{\Gamma_q(x)}$ and
$\Gamma_q\big(\fr{1-x}{x}\big)$}
\begin{theorem}\label{f1-f2-thm-1}
(a)\ The function $\Gamma_q\Big(\fr{1-x}{x}\Big)$ is strictly log-convex on $(0,\fr{1}{2}]$.

\noindent
(b)\ The function $\fr{\Gamma_q(1-x)}{\Gamma_q(x)}$ is strictly log-concave on $(0,\fr{1}{2}]$.
\end{theorem}
\begin{proof}
(a)\ Let $f_1(x) = \log \Gamma_q\Big(\fr{1-x}{x}\Big)$. Then
\[
\begin{split}
f_1'(x) & = -\fr{1}{x^2} \psi_q\big(\fr{1-x}{x} \big) \\
f_1''(x) & = \fr{2}{x^3}\psi_q\big(\fr{1-x}{x} \big) + \fr{1}{x^4}\psi_q'\big(\fr{1-x}{x} \big),
\end{split}
\]
and so,
$x^3 f_1''(x) = 2 \psi_q(y) + \fr{1}{x} \psi_q'(y)$ where $y = \fr{1-x}{x}$. Noting that $y< \fr{1}{x}$ and
that $\fr{1-x}{x} \geq 1$ on $(0,\fr{1}{2}]$, we get with the help of Lemma~\ref{lem-1}(b) that
\[
x^3 f_1''(x) > 2 \psi_q(y) + y \psi_q'(y) > 0.
\]
It follows that $f_1''(x) >0$ and thus $\Gamma_q\Big(\fr{1-x}{x}\Big)$ is strictly log-convex on $(0,\fr{1}{2}]$.

\noindent
(b)\ Let $f_2(x) = \log\fr{\Gamma_q(1-x)}{\Gamma_q(x)}$. Then $f_2''(x) = \psi_q'(1-x)-\psi_q'(x)$.
As $1-x >x$ on $(0,\fr{1}{2})$ and the function $\psi_q(x)$ is strictly decreasing by Lemma~\ref{lem-1}(a), we deduce
that $f_2''(x) <0$. This completes the proof.
\end{proof}
\noindent
For our next result we need the following lemma.
\begin{lemma}\label{f1-f2-lem-1}
There holds
\[
\psi_q\big(\fr{1}{2} \big) < 2 \psi_q(1).
\]
\end{lemma}
\begin{proof}
By virtue of Lemma~\ref{lem-q-psi-sum} applied to $n=2$, we have
$\psi_q\big(\fr{1}{2}\big) = \psi_q(1)- 2 \log(1+q^{\fr{1}{2}})$ and so our desired inequality means that
$\psi_q(1) > -2 \log(1+q^{\fr{1}{2}})$. Now, by a combination of Lemma~\ref{lem-2} applied to $x=1$ and the relation (\ref{psi-q-2}), we get
$\psi_q(1) > \fr{q\log q}{1-q}-2 \log(1+q^{\fr{1}{2}})$, which completes the proof.
\end{proof}
\begin{theorem}\label{f1-f2-thm-3}
For any $x\in(0,\fr{1}{2}]$, one has
\[
\fr{\Gamma_q(1-x)}{\Gamma_q(x)} \leq \Gamma_q\big(\fr{1-x}{x}\big),
\]
with equality only for $x=\fr{1}{2}$.
\end{theorem}
\begin{proof}
Let
\[
f(x)= \log \Gamma_q\big( \fr{1-x}{x} \big) +\log \Gamma_q(x) - \log \Gamma_q(1-x).
\]
As $f\big(\fr{1}{2}\big) = 0$, it will be enough to prove that $f(x)$ is decreasing on $(0,\fr{1}{2}]$.
Note first that from Theorem~\ref{f1-f2-thm-1}(a) we have that
$\Big(\log \Gamma_q\big( \fr{1-x}{x} \big)\Big)' $ is increasing on $(0,\fr{1}{2}]$, which implies that
\[
\Big(\log \Gamma_q\big( \fr{1-x}{x} \big)\Big)' = -\fr{1}{x^2}\psi_q\big(\fr{1-x}{x}\big) \leq -4 \psi_q(1).
\]
Moreover, since the function $\psi_q(x)$ is concave by Lemma~\ref{lem-1}(a), we have
\[
\psi_q(x) + \psi_q(1-x) \leq 2 \psi_q\big( \fr{x+(1-x)}{2} \big) = 2 \psi_q\big( \fr{1}{2}\big).
\]
Now as
\[
f'(x) = \Big(\log \Gamma_q\big( \fr{1-x}{x} \big)\Big)' + \psi_q(x) + \psi_q(1-x),
\]
we deduce from Lemma~\ref{f1-f2-lem-1} and the above facts that
\[
f'(x) \leq -4 \psi_q(1) + 2 \psi_q\big( \fr{1}{2} \big) = 2 \Big( \psi_q\big(\fr{1}{2}\big) - 2\psi_q(1)  \Big)
\leq 0.
\]
showing that $f(x)$ is decreasing. This completes the proof.
\end{proof}
\noindent
We close this section with an inequality of Ky Fan type for the $q$-gamma function.
For Ky Fan inequalities related to the classical gamma function the reader is referred to
Neuman~and~S\'{a}ndor~\cite{Neuman-Sandor}.
\begin{theorem}\label{KF-thm}
For a positive integer $k$ and $i=1,2,\ldots,k$, let $x_i\in (0,\fr{1}{2}]$ and let $x_i' = 1-x_i$. Let
$A_k$ denote the arithmetic mean of $x_i$ and let $A_k'$ denote the arithmetic mean of $x_i'$. Then
\[
\begin{split}
(a) \quad & \Gamma_q\big(\fr{A_k'}{A_k} \big) \leq \Big( \prod_{i=1}^k \Gamma_q\big(\fr{x_i}{x_i'} \big) \Big)^{\fr{1}{k}} \\
(b) \quad & \fr{\Gamma_q(A_k')}{\Gamma_q(A_k)} \geq \Big( \prod_{i=1}^k \fr{\Gamma_q(x_i')}{\Gamma_q(x_i)} \Big)^{\fr{1}{k}}.
\end{split}
\]
\end{theorem}
\begin{proof}
 As the function $\Gamma_q\Big(\fr{1-x}{x}\Big)$ is strictly log-convex on $(0,\fr{1}{2}]$ by Theorem~\ref{f1-f2-thm-1}(a), an
 application of Jensen's inequality to this function yields part (a). Moreover, an application of Jensen's inequality to
 $\fr{\Gamma_q(1-x)}{\Gamma_q(x)}$, which is strictly log-concave by Theorem~\ref{f1-f2-thm-1}(b), gives to part (b).
\end{proof}
\section{Inequalities related to $q$-series}
\begin{theorem}\label{thm-2-1}
For any positive integer $n$, any $x>0$, and any $a\in (0,1)$, we have
\[
\begin{split}
(a)\quad & (1-q^x)^{1-a} < \fr{(q^{x+a};q)_n}{(q^{x+1};q)_n} < (1-q^{x+a})^{1-a} \\
(b)\quad & (1-q^x)^{1-a} \leq {}_1\phi_0 \big(q^{a-1},-;q, q^{x+1}\big) \leq (1-q^{x+a})^{1-a}.
\end{split}
\]
\end{theorem}
\begin{proof}
Let $f(x)= (q^x;q)_n$. Then from $f(x) = \sum_{i=0}^{n-1} \log (1-q^{x+i})$, we get
\[
\Big(\log f(x)\Big)'' = -(\log q)^2 \sum_{i=0}^{n-1} \fr{q^{x+i}}{(1-q^{x+i})^2} < 0
\]
which means that the function $f(x)$ is strictly log-concave and so, $\fr{1}{f(x)}$ is strictly log-convex by Lemma~\ref{lem-3-2}(b). Then by Lemma~\ref{lem-3-1} applied to $\fr{1}{f(x)}$,
\[
(1-q^x)^{1-a}=\left(\fr{(q^x;q)_n}{(q^{x+1};q)_n}\right)^{1-a}  < \fr{(q^{x+a};q)_n}{(q^{x+1};q)_n} < \left(\fr{(q^{x+a};q)_n}{(q^{x+a+1};q)_n} \right)^{1-a} = (1-q^{x+a})^{1-a},
\]
which proves part (a). As to part (b), take limits as $n\to \infty$ in the previous inequalities and use the $q$-binomial theorem~\ref{q-binomial} to obtain
\[
(1-q^x)^{1-a} \leq \fr{(q^{x+a};q)_{\infty}}{(q^{x+1};q)_{\infty}} = {}_1\phi_0 \big(q^{a-1},-;q, q^{x+1}\big) \leq (1-q^{x+a})^{1-a},
\]
which is the desired double inequality.
\end{proof}
\begin{theorem}\label{thm-2-2}
For any positive integer $n$ we have
\[
\fr{1}{(q;q)_{\infty}} \leq \inf \left\{\fr{(q^{\fr{1}{2}};q)_{\infty}^{n-1}}
{(q;q^{\fr{1}{n}})_{\infty} (1-q^{\fr{1}{2}})^{n-1}},\
\fr{(q^{\fr{1}{2}};q)_{\infty}^{n-1}}
{(q^{\fr{1}{n}};q^{\fr{1}{n}})_{\infty}},\
\fr{(q^{\fr{1}{2}};q)_{\infty}^{n-1} (-q^{\fr{1}{n}}; q^{\fr{1}{n}})_{n-1}}
{(q^{\fr{1}{n}};q^{\fr{1}{n}})_{\infty} (1+q^{\fr{1}{2}})^{n-1}} \right\}
\]
\end{theorem}
\begin{proof}
The function $(1-q^x)\Gamma_q(x)$ is strictly log-convex by Askey~\cite{Askey}.
Then by Jensen's inequality
\[
\log\Big((1-q^{(x_1+\ldots+ x_k)/k})\Gamma_q(\frac{x_1+\ldots+ x_k}{k}) \Big)
\leq
\fr{1}{k}\Big(\log\big((1-q^{x_1})\Gamma_q(x_1)\big)+\ldots+\log\big((1-q^{x_k})\Gamma_q(x_k)\big)\Big),
\]
which by taking $k=n-1$ and $x_i = \frac{i}{n-1}$ and simplifying yield
\[
\Big((1-q^{1/2})\Gamma_q(1/2) \Big)^{n-1} \leq (q^{1/n};q^{1/n})_{n-1}
\prod_{i=1}^{n-1}\Gamma_q(i/n),
\]
or, by  (\ref{q-Gauss-prod-1}),
\[
\Big((1-q^{1/2})\Gamma_q(1/2) \Big)^{n-1} \leq
(q^{1/n};q^{1/n})_{n-1} \big(\Gamma_q(1/2)\big)^{n-1}
\frac{(q^{1/2};q^{1/2})_{\infty}^{n-1}}{(q;q)_{\infty}^{n-2} (q^{1/n};q^{1/n})_{\infty}}.
\]
Simplifying gives
\begin{equation}\label{help-1-thm-2-2}
\frac{(q^{1/2};q)_{\infty}^{n-1}}{(q;q^{1/n})_{\infty} (1-q^{1/2})^{n-1}} \geq \frac{1}{(q;q)_{\infty}}.
\end{equation}
Now apply Jensen's inequality to the strictly log-convex function $\Gamma_q(x)$
and proceed as before to obtain
\begin{equation}\label{help-2-thm-2-2}
\frac{(q^{1/2};q)_{\infty}^{n-1}}{(q^{1/n};q^{1/n})_{\infty}} \geq \frac{1}{(q;q)_{\infty}}.
\end{equation}
Finally, note that  $(1+q^x)\Gamma_q(x)$ is strictly log-convex and use the same sort of argument
as before to deduce that
\begin{equation}\label{help-3-thm-2-2}
\frac{(q^{1/2};q)_{\infty}^{n-1} (-q^{1/n};q^{1/n})_{n-1}}{(q^{1/n};q^{1/n})_{\infty} (1+q^{1/2})^{n-1}} \geq
\frac{1}{(q;q)_{\infty}}.
\end{equation}
Combining (\ref{help-1-thm-2-2}), (\ref{help-2-thm-2-2}), and (\ref{help-3-thm-2-2}) yields
the desired result.
\end{proof}
\begin{theorem}\label{thm-2-3}
For any integer $n>1$, we have
\[
 (q^{\fr{1}{2}};q)_{\infty}^{\varphi(n)} \geq
 \sup \left\{
 \prod_{d\mid n} (q^{\fr{1}{d}};q^{\fr{1}{d}})_{\infty}^{\mu(\fr{n}{d})},
(1-q^{\fr{1}{2}})^{\varphi(n)} \prod_{d\mid n} (q;q^{\fr{1}{d}})_{\infty}^{\mu(\fr{n}{d})},
(1+q^{\fr{1}{2}})^{\varphi(n)}
\fr{\prod_{d\mid n} (q^{\fr{1}{d}};q^{\fr{1}{d}})_{\infty}^{\mu(\fr{n}{d})}}
{\prod_{d\mid n} (-q^{\fr{1}{d}};q^{\fr{1}{d}})_{d-1}^{\mu(\fr{n}{d})}} \right\}.
\]
\end{theorem}
\begin{proof}
Note first the following well-known facts on the Euler totient function $\varphi(n)$:
\[
\sum_{\substack{i=1\\ (i,n)=1}}^n 1 = \varphi(n)\quad \text{and\quad}
\sum_{\substack{i=1\\ (i,n)=1}}^n i = \sum_{\substack{i=1\\ (i,n)=1}}^n (n-i)= \fr{n\varphi(n)}{2},
\]
from which it follows that
\[
\sum_{\substack{i=1\\ (i,n)=1}}^n \fr{i}{n} = \fr{\varphi(n)}{2} \quad\text{and\quad}
\fr{\displaystyle\sum_{\substack{i=1\\ (i,n)=1}}^n \fr{i}{n}}
{\displaystyle\sum_{\substack{i=1\\ (i,n)=1}}^n 1} = \fr{1}{2}.
\]
Apply Jensen's inequality to the  function $\Gamma_q(x)$ with $k=\varphi(n)$ and
$x_i= \fr{i}{n}$ for $i=1,\ldots,\varphi(n)$ and use the above to get
\[
\log \Gamma_q\Big( \fr{1}{2} \Big) \leq
\fr{1}{\varphi(n)} \log\prod_{\substack{i=1\\ (i,n)=1}}^n \Gamma_q\big(\fr{i}{n}\big),
\]
which by virtue of (\ref{q-short-gam-prod}) means
\[
\Big( \Gamma_q\Big(\fr{1}{2}\Big) \Big)^{\varphi(n)}
\leq P_q(n) = \fr{ \left(\Gamma_q\Big( \fr{1}{2} \Big)\right)^{\varphi(n)} (q^{\fr{1}{2}};q)^{\varphi(n)}}
{ \prod_{d\mid n}\big( \big(q^{\fr{1}{d}};q^{\fr{1}{d}}\big)_{\infty} \big)^{\mu\left(\fr{n}{d}\right)}}.
\]
It follows that
\begin{equation}\label{help-1-thm-2-3}
(q^{\fr{1}{2}};q)^{\varphi(n)} \geq  \prod_{d\mid n}
\big(q^{\fr{1}{d}};q^{\fr{1}{d}}\big)_{\infty}^{\mu\left(\fr{n}{d}\right)}.
\end{equation}
In the remaining part of the proof we shall need
\begin{equation}\label{help-2-thm-2-3}
\prod_{\substack{i=1\\ (i,n)=1}}^n (1-q^{\fr{i}{n}}) =
\prod_{d\mid n}(q^{\fr{1}{d}};q^{\fr{1}{d}})_{d-1}^{\mu\left(\fr{n}{d}\right)}
\end{equation}
which follows by the M\"{o}bius inversion formula applied to
$\prod_{i=1}^{n-1} (1-q^{\fr{i}{n}}) = (q^{\fr{1}{n}};q^{\fr{1}{n}})_{n-1}$.
Now apply Jensen's inequality to the  function $(1-q^x)\Gamma_q(x)$ with $k=\varphi(n)$ and
$x_i= \fr{i}{n}$ for $i=1,\ldots,\varphi(n)$  to deduce
\[
\log \Big((1-q^{\fr{1}{2}}) \Gamma_q\Big( \fr{1}{2}\Big) \Big) \leq
\fr{1}{\varphi(n)} \log\prod_{\substack{i=1\\ (i,n)=1}}^n (1-q^{\fr{i}{n}}) \Gamma_q\big(\fr{i}{n}\big),
\]
which by virtue of (\ref{q-short-gam-prod}) and (\ref{help-2-thm-2-3}) means
\[
\Big( (1-q^{\fr{1}{2}}) \Gamma_q\Big(\fr{1}{2}\Big) \Big)^{\varphi(n)}
\leq
\prod_{d\mid n}(q^{\fr{1}{d}};q^{\fr{1}{d}})_{d-1}^{\mu\left(\fr{n}{d}\right)}
 \fr{ \left(\Gamma_q\Big( \fr{1}{2} \Big)\right)^{\varphi(n)} (q^{\fr{1}{2}};q)^{\varphi(n)}}
{ \prod_{d\mid n}\big( \big(q^{\fr{1}{d}};q^{\fr{1}{d}}\big)_{\infty} \big)^{\mu\left(\fr{n}{d}\right)}}.
\]
Simplifying the foregoing inequality yields
\begin{equation}\label{help-3-thm-2-3}
(q^{\fr{1}{2}};q)^{\varphi(n)} \geq  (1-q^{\fr{1}{2}})^{\varphi(n)} \prod_{d\mid n}
\big(q;q^{\fr{1}{d}}\big)_{\infty}^{\mu\left(\fr{n}{d}\right)}.
\end{equation}
Furthermore, apply Jensen's inequality to the  function $(1+q^x)\Gamma_q(x)$ with $k=\varphi(n)$ and
$x_i= \fr{i}{n}$ for $i=1,\ldots,\varphi(n)$ and use the same argument as above to obtain
\begin{equation}\label{help-4-thm-2-3}
(q^{\fr{1}{2}};q)^{\varphi(n)} \geq
(1+q^{\fr{1}{2}})^{\varphi(n)}
\fr{\prod_{d\mid n} (q^{\fr{1}{d}};q^{\fr{1}{d}})_{\infty}^{\mu(\fr{n}{d})}}
{\prod_{d\mid n} (-q^{\fr{1}{d}};q^{\fr{1}{d}})_{d-1}^{\mu(\fr{n}{d})}}.
\end{equation}
Finally combine (\ref{help-1-thm-2-3}), (\ref{help-3-thm-2-3}), and (\ref{help-4-thm-2-3}) to
complete the proof.
\end{proof}
For our next result we need the following lemma of Askey~\cite{Askey} which deals with the behaviour of
$\Gamma_q$ as a function of $q$.
\begin{lemma}\label{lem-2-1}
Let $0<p<q<1$. Then
\begin{align*}
(a)\quad & \Gamma_p(x) \leq \Gamma_q(x) \leq \Gamma (x),& 0<x\leq 1\quad\text{or\quad} x\geq 2 \\
(b)\quad & \Gamma_p (x) \geq \Gamma_q(x) \geq \Gamma (x),& 1\leq x \leq 2.
\end{align*}
\end{lemma}
\begin{theorem}\label{thm-2-4}
Let $0<p<q<1$ and let $n>1$ be an integer. Then
\[
\begin{split}
(a)\quad \fr{(p;p)_{\infty}^n}{(p^{\fr{1}{n}};p^{\fr{1}{n}})_{\infty}} (1-p)^{\fr{n-1}{2}}
& \leq \fr{(q;q)_{\infty}^n}{(q^{\fr{1}{n}};q^{\fr{1}{n}})_{\infty}} (1-q)^{\fr{n-1}{2}}
 \leq \fr{(2\pi)^{\fr{n-1}{2}}}{\sqrt{n}} \\
(b)\quad \fr{(p;p)_{\infty}^n}{(p;p^{\fr{1}{n}})_{\infty}}  \fr{1}{(1-p)^{\fr{n-1}{2}}}
& \geq \fr{(q;q)_{\infty}^n}{(q;q^{\fr{1}{n}})_{\infty}} \fr{1}{(1-q)^{\fr{n-1}{2}}}
 \geq \fr{(n-1)! (2\pi)^{\fr{n-1}{2}}}{n^{n-\fr{1}{2}}}.
\end{split}
\]
\end{theorem}
\begin{proof}
Note first that
\begin{equation*}
\Big(\Gamma_q\big(\fr{1}{2}\big) \Big)^{n-1} (q^{\fr{1}{2}};q^{\fr{1}{2}})_{\infty}^{n-1} =
(1-q)^{\fr{n-1}{2}} (q;q)_{\infty}^{2n-2},
\end{equation*}
and therefore the relation (\ref{q-Gauss-prod-1}) boils down to
\begin{equation}\label{q-Gauss-prod-2}
\prod_{k=1}^{n-1}\Gamma_q\left(\fr{k}{n}\right) =
\fr{(q;q)_{\infty}^n}{(q^{\fr{1}{n}};q^{\fr{1}{n}})_{\infty}}(1-q)^{\fr{n-1}{2}}.
\end{equation}
(a)\ Let  $x_i = \fr{i}{n}$ for $i=1,\ldots, n-1$ and apply Lemma~\ref{lem-2-1}(a)  to obtain
\[
\prod_{i=1}^{n-1} \Gamma_p\big(\fr{i}{n}\big) \leq \prod_{i=1}^{n-1} \Gamma_q\big(\fr{i}{n}\big) \leq
\prod_{i=1}^{n-1} \Gamma\big(\fr{i}{n}\big) .
\]
Now use (\ref{q-Gauss-prod-2}) and (\ref{Gauss-prod}) to deduce
\[
\fr{(p;p)_{\infty}^n}{(p^{\fr{1}{n}};p^{\fr{1}{n}})_{\infty}} (1-p)^{\fr{n-1}{2}}
\leq \fr{(q;q)_{\infty}^n}{(q^{\fr{1}{n}};q^{\fr{1}{n}})_{\infty}} (1-q)^{\fr{n-1}{2}}
\leq \fr{(2\pi)^{{\fr{n-1}{2}}}}{\sqrt{n}},
\]
which is the desired inequalities.

\noindent
(b)\ As to this part let $x_i = 1+\fr{i}{n}$ for $i=1,\ldots, n-1$ and apply Lemma~\ref{lem-2-1} (b) to get
\[
\prod_{i=1}^{n-1} \Gamma_p\big(1+\fr{i}{n}\big) \geq \prod_{i=1}^{n-1} \Gamma_q\big(1+\fr{i}{n}\big) \geq
\prod_{i=1}^{n-1} \Gamma\big(1+ \fr{i}{n}\big) .
\]
It follows by combining these inequalities with the basic facts
$\Gamma_q(x+1) = \fr{1-q^x}{1-q} \Gamma_q (x)$ and $\Gamma(x+1) = x\Gamma(x)$
that
\[
\prod_{i=1}^{n-1} \Gamma_p\big(\fr{i}{n}\big)\fr{1-p^{\fr{i}{n}}}{1-p} \geq
\prod_{i=1}^{n-1} \Gamma_q\big(\fr{i}{n}\big)\fr{1-q^{\fr{i}{n}}}{1-q} \geq
\prod_{i=1}^{n-1} \Gamma\big(\fr{i}{n}\big) \fr{i}{n} ,
\]
or equivalently,
\[
\fr{(p^{\fr{1}{n}};p^{\fr{1}{n}})_{n-1}}{(1-p)^{n-1}}\fr{(p;p)_{\infty}^n}
{(p^{\fr{1}{n}};p^{\fr{1}{n}})_{\infty}} (1-p)^{\fr{n-1}{2}}
\geq
\fr{(q^{\fr{1}{n}};p^{\fr{1}{n}})_{n-1}}{(1-q)^{n-1}}\fr{(q;q)_{\infty}^n}
{(q^{\fr{1}{n}};q^{\fr{1}{n}})_{\infty}} (1-q)^{\fr{n-1}{2}}
\geq
\fr{(n-1)!}{n^{n-1}}  \prod_{i=1}^{n-1} \Gamma\big(\fr{i}{n}\big).
\]
Finally, an application of (\ref{q-Gauss-prod-2}) and (\ref{Gauss-prod}) to the foregoing inequalities
and simplifying yield
\[
\fr{(p;p)_{\infty}^n}{(p;p^{\fr{1}{n}})_{\infty}}\fr{1}{(1-p)^{\fr{n-1}{2}}} \geq
\fr{(q;q)_{\infty}^n}{(q;q^{\fr{1}{n}})_{\infty}}\fr{1}{(1-q)^{\fr{n-1}{2}}} \geq
\fr{(n-1)! (2\pi)^{\fr{n-1}{2}}}{n^{n-\fr{1}{2}}}.
\]
This completes the proof.
\end{proof}
\end{document}